\documentclass[12pt,a4paper]{article}
\usepackage{amsmath, amssymb, amsthm, setspace, verbatim}
\usepackage[all]{xy}
\theoremstyle{definition}
\newtheorem{thm}{Theorem}[section]
\newtheorem{defn}[thm]{Definition}
\newtheorem{prop}[thm]{Proposition}

\newtheorem{cor}[thm]{Corollary}

\newcommand{\goth}{\mathfrak}

\newcommand{\Gal}[1]{\mbox{Gal}(#1)}
\newcommand{\perm}[1]{\mbox{Perm}(#1)}

\def \OE {{\goth O}_{E}}
\def \OL {{\goth O}_{L}}
\def \OK {{\goth O}_{K}}

\def \p  {{\goth p}}
\def \P {{\goth P}}
\def \OEp {{\goth O}_{E,{\p}}}
\def \OEP {{\goth O}_{E,{\P}}}
\def \OLp {{\goth O}_{L,{\p}}}

\def \OKp {{\goth O}_{K,{\p}}}

\def \A {{\goth A}}

\def \AH {{\goth A}_{H}}
\def \AHp {\A_{H,\p}}
\def \M {{\goth M}}
\def \Mp {{\goth M}_{\p}}

\def \a {{\alpha}}

\def \e0 {e_{0}}

\title{Towards a Generalisation of Noether's Theorem to Nonclassical Hopf-Galois Structures}
\author{Paul J. Truman}

\begin{document}

\maketitle

\begin{abstract}
We study the nonclassical Hopf-Galois module structure of rings of algebraic integers in some extensions of $ p $-adic fields and number fields which are at most tamely ramified. We show that if $ L/K $ is an unramified extension of $ p $-adic fields which is $ H $-Galois for some Hopf algebra $ H $ then $ \OL $ is free over its associated order $  \AH $ in $ H $. If $ H $ is commutative, we show that this conclusion remains valid in ramified extensions of $ p $-adic fields if $ p $ does not divide the degree of the extension. By combining these results we prove a generalisation of Noether's theorem to nonclassical Hopf-Galois structures on domestic extensions of number fields. 
\end{abstract}

\section{Introduction}
Let $ L/K $ be a finite Galois extension of number fields or $ p $-adic fields (for some prime number $ p $) with group $ G $, and let $ \OL $ and $ \OK $ be the rings of algebraic integers or valuation rings of $ L $ and $ K $ respectively. By the normal basis theorem, $ L $ is a free module of rank one over the group algebra $ K[G] $. The ring of algebraic integers (or valuation ring) $ \OL $ is likewise a module over the integral group ring $ \OK[G] $, and Noether's theorem identifies when an analogous result holds at integral level: $ \OL $ is free over $ \OK[G] $ (for $ p $-adic fields) or {\em locally free} over $ \OK[G] $ (for number fields) if and only if $ L/K $ is at most tamely ramified \cite[Theorem 3]{Fro}. By locally free we mean that for each prime $ \p $ of $ \OK $ the completed ring of integers $ \OLp = \OKp \otimes_{\OK} \OL $ is free over the completed integral group ring $ \OKp[G] = \OKp \otimes_{\OK} \OK[G] $. An approach to studying wildly ramified extensions is to replace the integral group ring with a larger order in $ K[G] $, called the associated order:
\[ \A_{K[G]} = \{ \a \in K[G] \mid a \cdot x \in \OL \mbox{ for all } x \in \OL \}. \]
By construction $ \A_{K[G]} $ is the largest order in $ K[G] $ for which $ \OL $ is a module, and it is possible that $ \OL $ will be a free (or locally free) $ \A_{K[G]} $-module. In the $ p $-adic case Childs \cite{Ch87} provided a sufficient condition for this to occur by exploiting the fact that $ K[G] $ is a Hopf algebra - his theorem is that $ \OL $ is a free $ \A_{K[G]} $-module if the latter is a Hopf order in $ K[G] $. The action of the group algebra $ K[G] $ on a Galois extension $ L/K $ is a special case of the more general concept of a Hopf-Galois structure on a finite separable extension of fields. A given separable extension $ L/K $ may admit a number of Hopf-Galois structures, each consisting of a Hopf algebra $ H $ such that $ L $ is a $ H $-Galois extension of $ K $ (for the definition see the following section). If the extension is Galois then it admits at least one Hopf-Galois structure with Hopf algebra $ K[G] $, and we call this the classical structure. We call any other Hopf-Galois structures admitted by the extension nonclassical. A theorem of Greither and Pareigis reduces the enumeration of the Hopf-Galois structures admitted by a given extension to a group theoretic problem, and shows that the Hopf algebras all occur as ``twisted" forms of certain group algebras. To study the structure of $ \OL $ relative to the various Hopf-Galois structures admitted by the extension, we define within each Hopf algebra $ H $ an associated order:
\[ \AH = \{ h \in H \mid h \cdot x \in \OL \mbox{ for all } x \in \OL \}. \]
As with the group algebra $ K[G] $, for each Hopf algebra $ H $ which gives a Hopf-Galois structure on the extension, $ \AH $ is the largest order in $ H $ for which $ \OL $ is a module, and in fact $ \AH $ is the only order in $ H $ over which $ \OL $ can be free. Childs's theorem generalises to this context - if $ L/K $ is a finite $ H $-Galois extension of $ p $-adic fields and $ \AH $ is a Hopf order in $ H $ then $ \OL $ is a free $ \AH $-module. \\
\\
The use of nonclassical Hopf-Galois structures has proven to be fruitful in the study of wildly ramified extensions. For example, Byott \cite{By97a} has exhibited a class of wildly ramified Galois extensions $ L/K $ of $ p $-adic fields for which $ \OL $ is not free over $ \A_{K[G]} $, its associated order in the classical structure with Hopf algebra $ K[G] $, but is free over $ \AH $, its associated order in some Hopf algebra $ H $ giving a nonclassical structure on the extension. So from the point of view of describing $ \OL $, for these extensions the classical structure is not the ``correct" structure to use, and a nonclassical structure gives a more satisfactory description of the ring of algebraic integers. \\
\\
On the other hand, little is known about the nonclassical Hopf-Galois module structure of $ \OL $ when $ L/K $ is a tamely ramified extension. For Galois extensions, Noether's theorem states that in the classical structure with Hopf algebra $ K[G] $ we have $ \A_{K[G]} = \OK[G] $ and that $ \OL $ is free (for $ p $-adic fields) or locally free (for number fields) over $ \OK[G] $, and for number fields results such as the Hilbert-Speiser theorem describe the global structure of $ \OL $ over $ \OK[G] $ in certain cases \cite{HilSpe}. We might wonder whether analogous results hold for any nonclassical structures admitted by the extension. The purpose of this paper is to address the local question for certain classes of extensions which are at most tamely ramified. In sections 3 and 4 we prove the following two theorems concerning $ p $-adic fields: 
\begin{thm} \label{intro_theorem_unramified_local}
Let $ L/K $ be a finite unramified extension of $ p $-adic fields and let $ H $ be a Hopf algebra giving a Hopf-Galois structure on the extension. Then $ \OL $ is a free $ \AH $-module. 
\end{thm}
\begin{thm} \label{intro_theorem_maximal_local}
Let $ L/K $ be a finite (not necessarily Galois) extension of $ p $-adic fields and let  $ H $ be a commutative Hopf algebra giving a Hopf-Galois structure on the extension. Suppose that $ p \nmid [L:K] $. Then $ \OL $ is a free $ \AH $-module. 
\end{thm}
\noindent In section 5 we generalise these theorems slightly in order to study completions of extensions of number fields. We call a Galois extension $ L/K $ of number fields  {\em domestic} if no prime of $ \OK $ lying above a prime number dividing $ [L:K] $ ramifies in $ \OL $. By combining these generalised results we obtain the following analogue of Noether's theorem for nonclassical Hopf-Galois structures on domestic extensions: 
\begin{thm} \label{intro_cor_domestic}
Let $ L/K $ be a  finite abelian domestic extension of number fields and let $ H $ be a commutative Hopf algebra giving a Hopf-Galois structure on the extension. Then $ \OL $ is a locally free $ \AH $-module.
\end{thm}
\noindent In all of these cases we find that $ \AH $ has the same explicit description, connected to the theorem of Greither and Pareigis.

\section{Hopf-Galois Structures}

The notion of a Hopf-Galois structure is defined for certain extensions of commutative rings. We shall be interested mainly in studying Hopf-Galois structures on finite separable extensions of fields, but we give the definition in this more general context. Let $ R $ be a commutative ring with unity, $ S $ an $ R $-algebra which is finitely generated and projective as an $ R $-module, and $ H $ an $ R $-Hopf algebra which is finitely generated and projective as an $ R $-module. We shall write $ \varepsilon : H \rightarrow R $ for the counit of $ H $ and $ \Delta: H \rightarrow H \otimes_{R} H $ for the comultiplication of $ H $. We shall also make use of Sweedler notation
\[ \Delta(h) = \sum_{(h)} h_{(1)} \otimes h_{(2)}. \]
We say that $ S $ is an {\em $ H $-module algebra} if $ S $ is an $ H $-module and for all $ h \in H $ and $ s,t \in S $ we have

\begin{eqnarray*}
h \cdot(st) & = & \sum_{(h)} (h_{(1)} \cdot s) (h_{(2)} \cdot t) \\
h \cdot 1 & = & \varepsilon (h)1.
\end{eqnarray*} 

\begin{defn} \label{HGS_defn}
We say that $ S $ is an {\em $ H $-Galois extension of $ R $} ({\em $ H $-Galois} for short), or that $ H $ gives a {\em Hopf-Galois structure} on the extension, if $ S $ is an $ H $-module algebra and the $ R $-linear map
\[ j: S \otimes_{R} H \rightarrow End_{R}(S) \]
defined by
\[ j(s \otimes h)(t) = s(h \cdot t) \hspace{0.1in} \mbox{for} \hspace{0.1in} s,t \in S, \hspace{0.1in} h \in H \]
is an $ R $-module isomorphism. 
\end{defn}

\noindent A given finite separable extension of fields $ L/K $ may admit a number of Hopf-Galois structures. If the extension is Galois with group $ G $ then it admits at least the classical structure with Hopf algebra $ K[G] $. A theorem of Greither and Pareigis allows for the enumeration of all Hopf-Galois structures admitted by $ L/K $. Let $ E/K $ be the normal closure of $ L/K $. Let $ G = \Gal{E/K}, G^{\prime} = \Gal{E/L} $ and let $ X = \{ g G^{\prime}  \mid g \in G \} $ be the left coset space of $ G^{\prime} $ in $ G $. We shall write $ \overline{x} $ for the coset $ x G^{\prime} $, and $ \perm{X} $ for the group of permutations of the finite set $ X $. Define an embedding $ \lambda : G \rightarrow \perm{X} $ by left translation:
\[ \lambda(g)(\overline{x}) = \overline{gx} \mbox{  for } g \in G \mbox{ and } \overline{x} \in X. \]
Finally, we call a subgroup $ N $ of $ \perm{X} $ {\em regular} if $ |N| = |X| $ and $ N $ acts transitively on $ X $. We can now state the theorem of Greither and Pareigis:

\begin{thm}[{\bf Greither and Pareigis}] \label{GP_Theory}
There is a bijection between regular subgroups $ N $ of $ \perm{X} $ normalised by $ \lambda (G) $ and Hopf-Galois structures on $ L/K $. If $ N $ is such a subgroup, then $ G $ acts on the group algebra $ E[N] $ by acting simultaneously on the coefficients as the Galois automorphisms and on the group elements by conjugation via the embedding $ \lambda $. The Hopf algebra giving the Hopf-Galois structure corresponding to the subgroup $ N $ is 
\[ H = E[N]^{G} = \left\{ z \in E[N] \mid \,^{g}\!z = z \mbox{ for all } g \in G \right\}. \]
 Such a Hopf algebra then acts on the extension $ L/K $ as follows: if $ \displaystyle{  \sum_{n \in N} c_{n} n  \in H} $ (with $ c_{n} \in E $ a priori), then
\begin{equation} \label{GP_Action_Eqn}
\left( \sum_{n \in N} c_{n} n \right) \cdot x = \sum_{n \in N} c_{n} (n^{-1}(\overline{1_{G}}))x.
\end{equation}
\end{thm}
\begin{proof}
See \cite[Theorem 6.8]{ChTwe}.
\end{proof}

\noindent The Hopf algrebras produced by Theorem \ref{GP_Theory} inherit the Hopf algebra structure maps from the group algebra $ E[N] $, and are therefore cocommutative. Such a Hopf algebra is commutative precisely when the group $ N $ is abelian. Since all the fields we shall study have characteristic zero, all the Hopf algebras we shall study are separable $ K $-algebras (see \cite[(11.4)]{Wat}). The normal basis theorem generalises to $ H $-Galois extensions of fields: if $ L/K $ is such an extension then $ L $ is a free $ H $-module of rank one (see \cite[(2.16)]{ChTwe}). For extensions of local or global fields, it is natural to investigate analogous results at integral level. To study the structure of $ \OL $ relative to the Hopf-Galois structure with corresponding Hopf algebra $ H $ we define within $ H $ the associated order of $ \OL $:
\[ \AH = \{ h \in H \mid h \cdot x \in \OL \mbox{ for all } x \in \OL \}. \]
As noted in the introduction, $ \AH $ is the largest order in $ H $ for which $ \OL $ is a module. We are particularly interested in establishing whether $ \OL $ is a free (or locally free) $ \AH $-module. In the $ p $-adic case we have already mentioned Childs' theorem. We call an order $ \Lambda $ in a $ K $-Hopf algebra $ H $ a {\em Hopf order} if $ \Lambda $ is an $ \OK $-Hopf algebra with operations induced from $ K $.

\begin{thm}{\bf (Childs)} \label{childs_hopf}
Let $ L/K $ be a finite $ H $-Galois extension of $ p $-adic fields. If the associated order $ \AH $ is a Hopf order in $ H $, then $ \OL $ is a free $ \AH $-module. 
\end{thm} 
\begin{proof}
See \cite[Theorem 12.7]{ChTwe}. 
\end{proof}

\noindent We also state the following, which comes from integral representation theory. We recall that since we are concerned with fields of characteristic zero, a Hopf algebra $ H $ produced by Theorem \ref{GP_Theory} is separable. It follows (see \cite[Proposition 26.10]{CurRei}) that if $ H $ is commutative then it has a unique maximal order. 

\begin{prop} \label{maximal_order_free}
Let $ L/K $ be an $ H $-Galois extension of $ p $-adic fields for a commutative Hopf algebra $ H $. If $ \AH $ is the unique maximal order in $ H $ then $ \OL $ is a free $ \AH $-module. 
\end{prop}
\begin{proof}
Since $ \AH $ is the unique maximal order in $ H $, \cite[Theorem 26.12]{CurRei} implies that $ \OL $ is $ \AH $-projective. Since $ K $ is a $ p $-adic field and $ L $ is a free $ H $-module, we may apply \cite[Theorem 18.10]{MaxOrd}, and conclude that $ \OL $ is a free $ \AH $-module. 
\end{proof}

\noindent We end this section by considering a special order inside a Hopf algebra produced by the theorem of Greither and Pareigis (Theorem \ref{GP_Theory}).  Suppose that $ L/K $ is an extension of $ p $-adic fields or number fields with Galois closure $ E $, and that $ L/K $ is $ H $-Galois for some Hopf algebra $ H $. Then by Theorem \ref{GP_Theory}, we have that $ H = E[N]^{G} $ for $ N $ some regular subgroup of $ \perm{X} $ normalised by $ \lambda (G) $. Within this algebra, we shall study the order $ \OE[N]^{G} $.

\begin{prop} \label{Fixed_Points_Subset_A}
We have $ \OE[N]^{G} \subseteq \AH $.
\end{prop}
\begin{proof}
Let $ z \in \OE[N]^{G} $. Then $ z \in \OE[N] $, so we may write
\[ z =  \sum_{n \in N} c_{n} n \] 
with $ c_{n} \in \OE $. Since $ z \in H $, the action of $ z $ on an element $ x \in  L $ is given by equation (\ref{GP_Action_Eqn}).
Now for each $ n \in N $, any group element representing $ n^{-1}(\overline{1_{G}}) $ is a Galois automorphism of $ E $, so if $ x \in \OL $ then $ n^{-1}(\overline{1_{G}})x \in \OE $. Therefore for $ x \in \OL $ we have
\[ z \cdot x = \sum_{n \in N} c_{n} n^{-1}(\overline{1_{G}})x \in \OE. \]
Since also $ z \cdot x \in L $, we have that $ z \cdot x \in \OE \cap L = \OL $, whence $ z \in \AH $. 
\end{proof}

\noindent The proofs in this paper involve showing that under appropriate conditions we have locally the reverse inclusion. 

\section{Unramified Extensions}

Throughout this section, we let $ L/K $ be a finite unramified extension of $ p $-adic fields. Then $ L/K $ is automatically Galois, with cyclic Galois group, say $ G $. By Greither and Pareigis's theorem (Theorem \ref{GP_Theory}), a Hopf algebra $ H $ giving a Hopf-Galois structure on $ L/K $ is of the form $ L[N]^{G} $ for $ N $ some regular subgroup of $ \perm{G} $ normalised by $ \lambda (G) $. We shall show that $ \AH = \OL[N]^{G} $, and that this is a  Hopf order in $ H $, which by Theorem \ref{childs_hopf} implies that $ \OL $ is a free $ \AH $-module. We begin with a technical result:

\begin{prop}{\bf (Byott)} \label{Byott_Unramified}
Let $ X $ be a finite $ G $-set. Let $ \OL X $ denote the free $ \OL $-module on $ X $, with $ G $ acting via both $ \OL $ and $ X $. Then
\[ ( \OL X )^{G} \otimes_{\OK} \OL = \OL X. \]
\end{prop}
\begin{proof}
See \cite[Lemma 4.5]{By97}.
\end{proof}

\noindent We now use this result to show that if $ L/K $ is an unramified extension of $ p $-adic fields and $ H = L[N]^{G} $ is a Hopf algebra giving a Hopf-Galois structure on the extension then we have the reverse inclusion to Proposition \ref{Fixed_Points_Subset_A}.

\begin{prop} \label{PT_Unramified_Associated}
We have $ \AH = \OL[N]^{G} $. 
\end{prop}
\begin{proof}
By Proposition \ref{Fixed_Points_Subset_A}, $ \OL[N]^{G} \subseteq \AH. $ On the other hand, since $ L/K $ is unramified, we have that $ \OL \otimes_{\OK} \OL \cong \OL^{[L:K]} $, and this is the ring of integers of $ L \otimes_{K} L \cong L^{[L:K]} $. The group $ N $ acts on $ L^{[L:K]} $ by permuting the components, and so the $ L $-algebra $ H \otimes_{K} L \cong L[N] $ acts on $ L \otimes_{K} L $. The associated order of $ \OL \otimes_{\OK} \OL $ in $ L[N] $ is $ \OL[N] $. Since $ \AH \otimes_{\OK} \OL $ also acts on $ \OL \otimes_{\OK} \OL $, we conclude that 
\[ \AH \otimes_{\OK} \OL \subseteq \OL[N]. \]
So by Proposition \ref{Byott_Unramified} we have
\[ \AH \otimes_{\OK} \OL \subseteq \OL[N]^{G} \otimes_{\OK} \OL, \]
and therefore
\[ \AH \subseteq \OL[N]^{G}. \]
Hence $ \AH = \OL[N]^{G}. $
\end{proof}

\noindent We now use this explicit description of $ \AH $ to show that $ \AH $ is in fact a Hopf order. Since the Hopf algebras produced by Theorem \ref{GP_Theory} inherit the Hopf algebra structure maps from group algebras, it is sufficient to prove that the comultiplication on $ H $ restricts to $ \AH $.

\begin{prop} \label{PT_Hopf_Order}
The associated order $ \AH $ is a Hopf order in $ H $. 
\end{prop}
\begin{proof}
By Proposition \ref{PT_Unramified_Associated}, $\AH = \OL[N]^{G} $. We may reduce to the case where $ G $ acts regularly on some orbit $ X $ - in general this is not a group, but is still a finite $ G $-set. In this case, we have to show that $ \left( \OL X  \right)^{G} $ is an $ \OK $-coalgebra, and it is sufficient to show that we have
\[ \Delta \left( \left( \OL X \right)^{G} \right) \subseteq  \left( \OL X \right)^{G} \otimes \left( \OL X \right)^{G}. \]
Let $ a_{1}, \ldots ,a_{n} $ be a basis for $ \OL $ over $ \OK $, and fix some $ x \in X $. Then the elements
\[ b_{i} = \sum_{g \in G} g (a_{i}) (^{g}\!x ) \hspace{4mm} i = 1, \ldots ,n  \]
are a basis for $ (\OL X)^{G} $ over $ \OK $. For each $ i = 1, \ldots ,n $ we require that $ \Delta ( b_{i} ) \in  \left( \OL X \right)^{G} \otimes \left( \OL X \right)^{G}. $
Since $ \Delta $ is $ L $-linear, we have  
\[ \Delta ( b_{i} ) = \sum_{g \in G} g(a_{i}) (^{g}\!x \otimes\: ^{g}\!\!x) \hspace{4mm} i = 1, \ldots ,n. \]
This is fixed under the diagonal action of $ G  $ on $ LX \otimes_{L} LX $ since $ b_{i} \in H $. Additionally, since $ L/K $ is unramified  we have $ \det{(g(a_{i}))}^{2} \in \OK^{\times} $, and so by comparing with the basis $ \left\{ \left( \, ^{g}\!x \otimes \: \!^{g}\!x \right) \mid g \in G \right\} $ of $ \OL X \otimes_{\OL} \OL X $, we see that $ \Delta (b_{i}) \in \OL X \otimes_{\OL} \OL X $. Thus $ \Delta (b_{i}) \in \left( \OL X \right)^{G} \otimes_{\OK} \left( \OL X \right)^{G} $.
\end{proof}

\noindent In fact, in this case $ \OL[N]^{G} $ is the minimal Hopf order in $ H= L[N]^{G} $, since $ \OL[N] $ is the minimal Hopf order in $ L[N] $. We now restate and prove Theorem \ref{intro_theorem_unramified_local}:

\begin{thm} \label{theorem_unramified_local}
Let $ L/K $ be a finite unramified extension of $ p $-adic fields and let $ H = L[N]^{G} $ be a Hopf algebra giving a Hopf-Galois structure on the extension. Then $ \AH = \OL[N]^{G} $ and $ \OL $ is a free $ \AH $-module. 
\end{thm}
\begin{proof}
By Proposition \ref{PT_Unramified_Associated} we have $\AH = \OL[N]^{G} $, and by Proposition \ref{PT_Hopf_Order} this is a Hopf order in $ H $. Now apply Theorem \ref{childs_hopf}. 
\end{proof}

\section{Maximal Associated Orders}

Throughout this section we let $ L/K $ be a finite (not necessarily Galois) extension of $ p $-adic fields with Galois closure $ E $, and suppose that $ p \nmid [L:K] $. We shall consider Hopf-Galois structures admitted by $ L/K $ for which the corresponding Hopf algebra $ H $ is commutative. We recall that since $ K $ has characteristic zero, $ H $ is a separable $ K $-algebra, and so this implies that $ H $ has a unique maximal order. In the notation established prior to Theorem \ref{GP_Theory}, we have that $ H = E[N]^{G} $ for some abelian regular subgroup $ N $ of $ \perm{X} $ normalised by $ \lambda(G) $. In particular we note that $ |N| = [L:K] $. We show that in this case $ \AH $ coincides with the unique maximal order in $ H $. When this occurs, it follows by Proposition \ref{maximal_order_free} that $ \OL $ is a free $ \AH $-module.  

\begin{prop} \label{Local_Group_Ring_Maximal_Order}
The integral group ring $ \OE[N] $ is the unique maximal order in $ E[N] $. 
\end{prop}
\begin{proof}
The group algebra $ E[N] $ is separable and commutative and therefore has a unique maximal order (the integral closure of $ \OE $ in $ E[N] $). If $ \Lambda $ is any order in $ E[N] $ containing $ \OE[N] $ then we have by \cite[Proposition 27.1]{CurRei} that
\[ \OE[N] \subseteq \Lambda \subseteq |N|^{-1} \OE[N], \] 
so $ \OE[N] $ is maximal since $ |N| \in \OE^{\times} $. 
\end{proof}

\noindent We now show that taking the fixed points of $ E[N] $ under the action by $ G $ preserves this maximality, so that $ \OE[N]^{G} $ is the unique maximal order in $ H = E[N]^{G} $. 

\begin{prop} \label{Fixed_Points_Maximal}
Let $ G $ act on the group algebra $ E[N] $ by acting on $ E $ as Galois automorphisms and on  $ N $ by conjugation via the embedding $ \lambda $. Then $ \OE[N]^{G} $ is the unique maximal order in the $ K $-algebra $ E[N]^{G} $. 
\end{prop}
\begin{proof}
Since $ E $ has characteristic zero and $ N $ is abelian, $ E[N]^{G} $ has a unique maximal order. Since $ p \nmid |N| $, the maximal $ \OE $ order in $ E[N] $ is $ \OE[N] $ by Proposition \ref{Local_Group_Ring_Maximal_Order}. Denote by $ {\goth M} $ the maximal order in $ E[N]^{G} $, and let $ x \in {\goth M} $. Then $ x $ is integral over $ \OK $ in $ E[N]^{G} $, so $ x $ is integral over $ \OE $ in $ E[N] $, whence $ x \in \OE[N] $. So $ x \in E[N]^{G} \cap \OE[N] = \OE[N]^{G} $, and so $ \OE[N]^{G} = {\goth M} $. 
\end{proof}

\begin{prop} \label{AH_Maximal}
The associated order $ \AH $ is the unique maximal order in $ H $. 
\end{prop}
\begin{proof}
By Proposition \ref{Fixed_Points_Maximal}, $ \OE[N]^{G} $ is the unique maximal order in $ H $. On the other hand, by Proposition \ref{Fixed_Points_Subset_A} $ \OE[N]^{G} \subseteq \AH $. So $ \OE[N]^{G}= \AH $ is the unique maximal order in $ H $. 
\end{proof}

\noindent We now restate and prove Theorem \ref{intro_theorem_maximal_local}

\begin{thm} \label{theorem_maximal_local}
Let $ L/K $ be a finite (not necessarily Galois) extension of $ p $-adic fields and let  $ H $ be a commutative Hopf algebra giving a Hopf-Galois structure on the extension. Suppose that $ p \nmid [L:K] $. Then $ \AH = \OL[N]^{G} $ and $ \OL $ is a free $ \AH $-module. 
\end{thm}
\begin{proof}
By Proposition \ref{AH_Maximal} we have that $ \AH = \OE[N]^{G} $ and that this is the unique maximal order in $ H $. Now apply Proposition \ref{maximal_order_free}. 
\end{proof}

\section{Consequences For Number Fields}

In this section we consider a finite extension of number fields $ L/K $. We shall prove results analogous to those in sections 3 and 4 which will give us information about the local structure of $ \OL $ as a module over its associated order $ \AH $ in a Hopf algebra $ H $ giving a Hopf-Galois structure on the extension. \\
\\
\noindent If $ \p $ is a prime of $ \OK $ and $ A $ is a $ K $-algebra then we shall write $ A_{\p} $ for the $ K_{\p} $-algebra $ A \otimes_{K} K_{\p} $, and similarly for orders in $ A $. We then have that $ L_{\p} $ is an $ H_{\p} $-Galois extension of $ K_{\p} $, and we seek to study the completed ring of integers $ \OLp $ over the completed associated order $ \AHp $. In general $ L_{\p} $ is not a local field but a finite product of local fields - we have the isomorphism
\[ L_{\p} \cong \prod_{\P \mid \p} L_{\P}, \]
where the product is taken over the prime ideals $ \P $ of $ \OL $ which lie above $ \p $ and each $ L_{\P} $ is a $ p $-adic field. We have an analogous decomposition at integral level. (see \cite[(2.16)]{FroTay}.) \\
\\
\noindent Since the results quoted in Theorem \ref{childs_hopf} and Proposition \ref{maximal_order_free} are applicable only to extensions of local fields, we require generalisations of these results in order to proceed. The appropriate generalisation of Proposition \ref{maximal_order_free} is straightforward:

\begin{prop} \label{maximal_order_free_global}
Let $ L/K $ be an extension of number fields which is $ H $-Galois for a commutative Hopf algebra $ H $, and let $ \p $ be a prime of $ \OK $. If $ \AHp $ is the unique maximal order in $ H_{\p} $ then $ \OLp $ is a free $ \AHp $-module. 
\end{prop}
\begin{proof}
This is essentially the same as the proof of Proposition \ref{maximal_order_free}. 
\end{proof}

\noindent To state the appropriate generalisation of Childs' theorem (Theorem  \ref{childs_hopf}) we need a generalisation of the notion of tameness, due to Childs (\cite[(13.1)]{ChTwe}). Let $ H $ be a Hopf algebra (over an arbitrary commutative ring $ R $) and $ S $ an  $ R $-algebra which is finitely generated and projective as an $ R $-module, and which is an $ H$-module algebra. We call an element $ \theta \in H $ a {\em left integral} if for all $ h \in H $ we have $ h \theta  = \varepsilon (h) \theta $, where $ \varepsilon : H \rightarrow R $ is the counit map. We say that $ S $ is an {\em $ H $-tame extension of $ R $} if 
\begin{enumerate}
\item $ \{ s \in S \mid hs = \varepsilon(h) s \mbox{ for all } h \in H \} = R $.
\item $ \mbox{rank}_{R}(S) = \mbox{rank}_{R}(H) $. 
\item $ S $ is a faithful $ H $-module. 
\item There exists a left integral $ \theta $ of $ H $ satisfying $ \theta S = R $. 
\end{enumerate}

\noindent Then we have:

\begin{prop}
If $ \AHp $ is a Hopf order in $ H_{\p} $ and $ \OLp $ is an $ \AHp $-tame extension of $ \OKp $ then $ \OLp $ is a free $ \AHp $-module. 
\end{prop} 
\begin{proof}
See \cite[Theorem 13.4]{ChTwe}. 
\end{proof}

\noindent We can now prove analogues of the results in sections 3 and 4 for completions of extensions of number fields. We begin with completions at an unramified prime $ \p $. Motivated by Proposition \ref{PT_Unramified_Associated} we consider the $ \OKp $-order $ \OLp[N]^{G} $ in the completed Hopf algebra $ H_{\p} $. Note that in this proposition we do not require that $ H $ be commutative. 

\begin{prop} \label{PT_Hopf_Order_Global}
Let $ L/K $ be a finite abelian extension of number fields with group $ G $, and suppose that $ L/K $ is $ H $-Galois for the Hopf algebra $ H = L[N]^{G} $. Let $ \p $ be a prime of $ \OK$ which is unramified in $ \OL $. Then the order $ \OLp[N]^{G} $ is a Hopf order in $ H_{\p} $. 
\end{prop}
\begin{proof}
This follows the proof of Proposition \ref{PT_Hopf_Order}. Note that in this case we have explicitly assumed that $ G $ is abelian, so any faithful transitive action of a subgroup or quotient group of $ G $ is regular. 
\end{proof}

\begin{thm} \label{theorem_unramified_global}
Let $ L/K $ be a finite abelian extension of number fields with group $ G $, and suppose that $ L/K $ is $ H $-Galois for the Hopf algebra $ H = L[N]^{G} $. Let $ \p $ be a prime of $ \OK$ which is unramified in $ \OL $. Then $ \AHp = \OLp[N]^{G} $ and $ \OLp $ is a free $ \AHp $-module. 
\end{thm}
\begin{proof}
By Proposition \ref{PT_Hopf_Order_Global}, $ \OLp[N]^{G} $ is a Hopf order in $ H_{\p} $. We note that the trace element
\[ \theta= \sum_{n \in N} n \]
is a left integral of $ \OLp[N]^{G} $, and since $ \p $ is unramified in $ \OL $ there exists an element $ t \in \OLp $ such that $ \theta \cdot t = 1 $. Thus $ \OLp $ is an $ \OLp[N]^{G} $-tame extension of $ \OKp $, and so $ \OLp $ is a free $ \OLp[N]^{G} $-module. Thus $ \AHp = \OLp[N]^{G} $. 
\end{proof}

\begin{cor} \label{cor_unramified}
Under the same assumptions as Theorem \ref{theorem_unramified_global}, $ \OLp $ is free over $ \AHp $ for all primes $ \p $ of $ \OK $ which are unramified in $ \OL $. Thus in order to determine whether $ \OL $ is a locally free $ \AH $-module, it is sufficient to consider the structure of $ \OLp $ over $ \AHp $ for each of the (finitely many) primes $ \p $ which are ramified in $ \OL $. 
\end{cor}

\noindent Now we consider the situation analogous to that considered in section 4. 

\begin{prop} \label{global_fixed_points_maximal}
Let $ L/K $ be a finite (not necessarily Galois) extension of number fields with Galois closure $ E $. Suppose $ L/K $ is $ H $-Galois for some commutative Hopf algebra $ H = E[N]^{G} $. Let $ \p $ be a prime of $ \OK $ which lies above a prime number $ p \nmid [L:K] $. Then $ \OEp[N]^{G} $ is the unique maximal order in $ H_{\p} $. 
\end{prop}
\begin{proof}
Let $ \M $ denote the unique maximal order in $ H $, so that $ \Mp $ is the unique maximal order in $ H_{\p} $, and let $ x\in \Mp $. Then $ x \in E_{\p}[N]^{G} $, so $ x \in E_{\p}[N] $. We have an isomorphism 
\[ E_{\p}[N] \cong \prod_{\P \mid \p} E_{\P}[N], \]
where the product is taken over the prime ideals $ \P $ of $ \OL $ lying abve $ \p $, and each factor on the right is a group algebra over a $ p $-adic field whose residue characteristic is coprime to  $ |N| $. Applying Proposition \ref{Local_Group_Ring_Maximal_Order} to each factor on the right, we see that the image of $ x $ under the isomorphism above lies in the product
\[ \prod_{\P \mid \p} \OEP[N] \cong \OEp[N], \]
and so $ x \in \OEp[N] \cap E_{\p}[N]^{G} = \OEp[N]^{G} $. Therefore $ \Mp = \OEp[N]^{G}  $. 
\end{proof}

\begin{prop} \label{AH_Maximal_global}
Retain the assumptions of Proposition \ref{global_fixed_points_maximal}. Then the completed associated order $ \AHp $ is the unique maximal order in $ H_{\p} $. 
\end{prop}
\begin{proof}
By Proposition \ref{global_fixed_points_maximal}, $ \OEp[N]^{G} $ is the unique maximal order in $ H_{\p} $. On the other hand, by Proposition \ref{Fixed_Points_Subset_A} $ \OEp[N]^{G} \subseteq \AHp $. So $ \OEp[N]^{G}= \AHp $ and this is the maximal order in $ H_{\p} $. 
\end{proof}

\begin{thm} \label{theorem_maximal_global}
Let $ L/K $ be a finite (not necessarily Galois) extension of number fields with Galois closure $ E $, and let $ G = \Gal{E/K} $. Suppose that $ L/K $ is $ H $-Galois for some commutative Hopf algebra $ H = E[N]^{G} $. Suppose that $ \p $ is a prime of $ \OK $ lying above a prime number $ p \nmid [L:K] $. Then $ \AHp = \OEp[N]^{G} $ and $ \OLp $ is a free $ \AHp $-module. 
\end{thm}
\begin{proof}
By Proposition \ref{global_fixed_points_maximal} and Proposition \ref{AH_Maximal_global} we have that $ \AHp = \OLp[N]^{G} $ and that this is the unique maximal order in $ H_{\p} $. Now apply Proposition \ref{maximal_order_free_global}. 
\end{proof}

\noindent We obtain Theorem \ref{intro_cor_domestic} by combining these results. Recall that a Galois extension $ L/K $ of number fields is called {\em domestic} if no prime of $ \OK $ lying above a prime number dividing $ [L:K] $ ramifies in $ \OL $. 

\begin{thm} \label{cor_domestic}
Let $ L/K $ be a  finite domestic abelian Galois extension of number fields. Suppose that $ L/K $ is $ H $-Galois for some commutative Hopf algebra $ H $. Then $ \AH = \OL[N]^{G} $ and $ \OL $ is a locally free $ \AH $-module.
\end{thm}
\begin{proof}
By Theorem \ref{theorem_unramified_global}, we have that if $ \p $ is a prime of $ \OK $ which is unramified in $ \OL $ then $ \AHp = \OLp[N]^{G} $ and $ \OLp $ is a free $ \AHp $-module. Suppose $ \p $ is a prime of $ \OK $ which is ramified in $ \OL $. Then $ \p $ lies above a prime number $ p $, and since $ L/K $ is domestic, we have $ p \nmid [L:K] $. We may therefore apply Theorem \ref{theorem_maximal_global}, and conclude that $ \AHp = \OLp[N]^{G} $ and that $ \OLp $ is a free $ \AHp $-module.
\end{proof}

\noindent As a particular example, we have:

\begin{cor} \label{prime_power_number_field_free}
Let $ L/K $ be an abelian Galois extension of number fields of prime power degree which is at most tamely ramified. Suppose that $ L/K $ is $ H $-Galois for some commutative Hopf algebra $ H $. Then $ \OL $ is a locally free $ \AH $-module.
\end{cor}
\begin{proof}
By Corollary \ref{cor_domestic}, it is sufficient to observe that since $ L/K $ has prime power degree, the assumption that it is tamely ramified is equivalent to the assumption that it is domestic.
\end{proof}

\bibliography{refs}

\begin{thebibliography}{Byo97b}

\bibitem[Byo97a]{By97a}
N.~P. Byott.
\newblock Galois structure of ideals in wildly ramified abelian $p$-extensions
  of a $p$-adic field, and some applications.
\newblock {\em Journal de Theorie des Nombres de Bordeaux}, 9:201--219, 1997.

\bibitem[Byo97b]{By97}
N.~P. Byott.
\newblock Tame realisable classes over {H}opf orders.
\newblock {\em Journal of Algebra}, 201:284--316, 1997.

\bibitem[Chi87]{Ch87}
L.~N. Childs.
\newblock Taming wild extensions with {H}opf algebras.
\newblock {\em Trans. Amer. Math. Soc.}, 304:111--140, 1987.

\bibitem[Chi00]{ChTwe}
L.~N. Childs.
\newblock {\em Taming Wild Extensions: {H}opf Algebras and local {G}alois
  module theory}.
\newblock American Mathematical Society, 2000.

\bibitem[CR81]{CurRei}
C.~W. Curtis and I.~Reiner.
\newblock {\em Methods of Representation Theory with Applications to Finite
  Groups and Orders}, volume~1.
\newblock Wiley, 1981.

\bibitem[Fr{\"{o}}83]{Fro}
A.~Fr{\"{o}}hlich.
\newblock {\em Galois Module Structure of Algebraic Integers}.
\newblock Springer, 1983.

\bibitem[FT91]{FroTay}
A.~Fr{\"{o}}hlich and M.~J. Taylor.
\newblock {\em Algebraic Number Theory}.
\newblock Cambridge University Press, 1991.

\bibitem[Hil65]{HilSpe}
D.~Hilbert.
\newblock {\em Die Theorie der algebraischen Zahlen}, volume~1.
\newblock 1965.

\bibitem[Rei75]{MaxOrd}
I.~Reiner.
\newblock {\em Maximal Orders}.
\newblock Academic Press, 1975.

\bibitem[Wat97]{Wat}
W.C. Waterhouse.
\newblock {\em Introduction to Affine Group Schemes}.
\newblock Springer, 1997.

\end{thebibliography}
\bibliographystyle{alpha}

\end{document}